\documentclass[12pt,a4paper,oneside]{amsart}

\usepackage{amsmath} 
\usepackage{amssymb} 
\usepackage{amsthm} 
\usepackage[english]{babel} 
\usepackage[font=small,justification=centering]{caption} 
\usepackage[nodayofweek]{datetime}
\usepackage[T1]{fontenc} 
\usepackage[a4paper]{geometry} 
\usepackage[utf8]{inputenc} 
\usepackage{ifthen} 
\usepackage{mathabx} 
\usepackage{mathtools} 
\usepackage[dvipsnames]{xcolor} 
\usepackage{setspace} 
\usepackage{tikz-cd} 
\usepackage{xfrac} 
\usepackage{cleveref}
\usepackage{crossreftools} 
\usepackage{stmaryrd} 
\usepackage{comment} 
\usepackage{geometry}
\usepackage{enumerate}
\usepackage{enumitem} 
\usepackage[foot]{amsaddr}

\makeatletter
\def\@tocline#1#2#3#4#5#6#7{\relax
  \ifnum #1>\c@tocdepth 
  \else
    \par \addpenalty\@secpenalty\addvspace{#2}%
    \begingroup \hyphenpenalty\@M
    \@ifempty{#4}{%
      \@tempdima\csname r@tocindent\number#1\endcsname\relax
    }{%
      \@tempdima#4\relax
    }%
    \parindent\z@ \leftskip#3\relax \advance\leftskip\@tempdima\relax
    \rightskip\@pnumwidth plus4em \parfillskip-\@pnumwidth
    #5\leavevmode\hskip-\@tempdima
      \ifcase #1
       \or\or \hskip 1em \or \hskip 2em \else \hskip 3em \fi%
      #6\nobreak\relax
    \dotfill\hbox to\@pnumwidth{\@tocpagenum{#7}}\par
    \nobreak
    \endgroup
  \fi}
\makeatother

\theoremstyle{definition}

\DeclareMathOperator{\diam}{diam}

\newcommand{\length}{\textrm{length}}

\usepackage{tikz}
\usetikzlibrary{arrows,quotes}
\tikzstyle{blackNode}=[fill=black, draw=black, shape=circle]

\newtheorem*{proposition*}{Proposition}
\newtheorem*{notation*}{Notation}
\newtheorem*{lemma*}{Lemma}
\newtheorem*{definition*}{Definition}

\newtheorem{theorem}{Theorem}

\newtheorem{corollary}[theorem]{Corollary}

\newtheorem{lemma}[theorem]{Lemma}

\newtheorem{proposition}[theorem]{Proposition}

\newtheorem*{remark*}{Remark}

\newtheorem*{question*}{Question}

\newtheorem*{example*}{Example}

\newtheorem*{claim*}{Claim}

\newtheorem*{conjecture*}{Conjecture}

\newtheorem{manualtheoreminner}{Theorem}
\newenvironment{manualtheorem}[1]{%
  \renewcommand\themanualtheoreminner{#1}%
  \manualtheoreminner
}{\endmanualtheoreminner}

\newcommand\bolden[1]{{\boldmath\bfseries#1}}

\usepackage{physics}
\def\Z{{\mathbb Z}}
\def\N{{\mathbb N}}
\def\Q{{\mathbb Q}}

\def\R{{\mathbb R}}
\def\hyp{{\mathbb H}}
\def\XX{{\mathcal{X}}}

\newcommand{\sol}{\textrm{Sol}}
\newcommand{\dl}{\textrm{DL}}
\newcommand{\summ}{\textrm{sum}}
\newcommand{\con}{\textrm{Con}}

\DeclarePairedDelimiter\absval{\lvert}{\rvert}

\title{Embeddings of Trees, Cantor Sets and Solvable Baumslag--Solitar Groups}
\author{Patrick S. Nairne}
\email{nairne@maths.ox.ac.uk}

\address{Mathematical Institute, University of Oxford}

\begin{document}

\raggedbottom

\maketitle

\begin{abstract}
We characterise when there exists a quasiisometric embedding between two solvable Baumslag--Solitar groups. This extends the work of Farb and Mosher on quasiisometries between the same groups. More generally, we characterise when there can exist a quasiisometric embedding between two treebolic spaces. This allows us to determine when two treebolic spaces are quasiisometric, confirming a conjecture of Woess. The question of whether there exists a quasiisometric embedding between two treebolic spaces turns out to be equivalent to the question of whether there exists a bilipschitz embedding between two symbolic Cantor sets, which in turn is equivalent to the question of whether there exists a rough isometric embedding between two regular rooted trees. Hence we answer all three of these questions simultaneously. It turns out that the existence of such embeddings is completely determined by the boundedness of an intriguing family of integer sequences.
\end{abstract}

\tableofcontents

\section{Introduction}

In his speech to the ICM in 1983 \cite{ICM}, Gromov proposed the vast project of classifying finitely generated groups up to quasiisometry. His line of thought leads to the question of whether an algebraic property $P$ of finitely generated groups can be characterised by some other geometric (i.e. quasiisometry invariant) property. If this is the case then we say that the class of groups satisfying $P$ is \textit{quasiisometrically rigid}. It suggests that the algebraic property $P$ in fact corresponds to some geometric feature of the groups. For example, Gromov \cite{GROMOV} proved that virtual nilpotency is invariant under quasiisometries and corresponds to the geometric property of polynomial growth. In contrast, Erschler \cite{ERSCHLER} proved that virtual solvability is not geometric; there exist groups $G$ and $H$ where $G$ is solvable and $H$ is not virtually solvable such that $G$ and $H$ have a common Cayley graph. Nevertheless, we can limit our focus to subclasses of solvable groups. Farb and Mosher \cite{FM} proved that the solvable Baumslag--Solitar groups $BS(1,m)$ exhibit a fascinating rigidity; $BS(1,m)$ is quasiisometric to $BS(1,n)$ if and only if $m,n$ are powers of a common integer (which holds if and only if they are commensurable). This paper generalises this result in two ways. First, the result is extended to the \textit{treebolic spaces} $HT(p,q)$; in their paper, Farb and Mosher utilise the fact that $BS(1,m)$ acts cocompactly and isometrically on $HT(m,m)$. Treebolic spaces were introduced by Bendikov, Saloff--Coste, Salvatori and Woess in \cite{tree1} and they developed the theory further in \cite{tree2, tree3}. We characterise when treebolic spaces are quasiisometric (\Cref{thm2}), confirming a conjecture of Woess \cite{WOESS}. Second, we study a stronger notion of rigidity by providing necessary and sufficient conditions for the existence of a quasiisometric embedding of one treebolic space into another (\Cref{thm1}). In particular, we prove that $BS(1,m)$ quasiisometrically embeds into $BS(1,n)$ if and only if $m,n$ are powers of a common integer. 

\subsection{Statement of results}

Let $p \in \N_{\geq 2}$ and let $q > 1$. We are primarily interested in three classes of metric space: \textit{regular rooted trees} $R(p,q)$, \textit{symbolic Cantor sets} $\Z(p,q)$ and \textit{treebolic spaces} $HT(p,q)$. These can be described as follows.

\begin{itemize}
\item Let $R(p,q)$ denote the rooted metric tree with all edges of length $\log(q)$ and with valency $p+1$ at every vertex apart from the basepoint which has valency $p$.
\item Let $\Z(p,q) = \{ 0, 1, ..., p-1 \}^{\N}$ be the space of infinite sequences on $p$ letters. $\Z(p,q)$ admits a metric $\rho$ making it a metric space of diameter $1$. See Section 2.
\item Let $T(p,q)$ be the metric tree with all edges of length $\log(q)$ and with valency $p+1$ at every vertex. Roughly speaking, the treebolic space $HT(p,q)$ is formed by fixing height functions on $T(p,q)$ and on the hyperbolic plane $\hyp^2$ and then gluing horostrips of $\hyp^2$ onto every edge of $T(p,q)$ in a height-preserving manner. One can also think of $HT(p,q)$ as being composed of infinitely many copies of $\hyp^2$ glued together along horoballs into the shape of a tree. See Section 2 for a rigorous definition.
\end{itemize}

We will be interested in three different types of embedding between these metric spaces: rough isometric embeddings, bilipschitz embeddings and quasiisometric embeddings. Fix a pair of metric spaces $X$, $Y$ and a pair of constants $K \geq 1$, $C \geq 0$. We define $(K,C)$-quasiisometric embeddings and $(K,C)$-quasiisometries in the usual manner. A $(K,0)$-quasiisometric embedding $f: X \rightarrow Y$ is a \textit{$K$-bilipschitz embedding}. A $(K,0)$-quasiisometry $f: X \rightarrow Y$ is a \textit{$K$-bilipschitz homeomorphism}. A $(1,C)$-quasiisometric embedding $f: X \rightarrow Y$ is a \textit{$C$-rough isometric embedding}. A $(1,C)$-quasiisometry $f: X \rightarrow Y$ is a \textit{$C$-rough isometry}.

We will simultaneously answer three closely connected questions: when does there exists a rough isometric embedding $R(p,q) \rightarrow R(p',q')$? When does there exists a bilipschitz embedding $\Z(p,q) \rightarrow \Z(p',q')$? And, finally, when does there exist a quasiisometric embedding $HT(p,q) \rightarrow HT(p',q')$? Remarkably, the existence of these embeddings is intimately related to the boundedness of a certain integral sequence $\XX = \XX(p,q,p',q')$ which can be defined as follows. 

\begin{itemize}
\item Consider the non-negative number line $\R_{\geq 0}$ and imagine marking each non-negative multiple of $\log(q)$ in blue and each non-negative multiple of $\log(q')$ in red. We begin with a single pebble in our hand. We then imagine walking forwards along the number line from $0$ obeying the following rule as we go: each time we pass a blue, we multiply the amount of pebbles in our possession by $p$; each time we pass a red, we divide our pebbles evenly into $p'$ groups and keep only one of the larger groups (e.g. if $p' = 3$ then we divide 8 pebbles into 3,3,2 and keep only 3 pebbles); if we pass a red and a blue simultaneously then we first multiply our pebbles by $p$ and only then do we divide them by $p'$ and keep one of the larger groups as before. The changing quantity of pebbles in our possession as we walk along the number line forms the integral sequence $\XX(p,q,p',q')$. For a formal definition, see the definition of the more general class of sequences $\XX((p_n),(q_n),(p_n'),(q_n'))$ in Section 3. 
\end{itemize}

Consider the following pair of number theoretic conditions on $p,q,p',q'$.
\begin{enumerate}[label=(C\arabic*)] 
    \item $\log(p)/\log(q) < \log(p')/\log(q')$;
    \item $\log(p)/\log(q) = \log(p')/\log(q')$ and $p, p'$ are powers of a common integer. 
\end{enumerate}

We prove the following series of equivalences. 

\begin{theorem} \label{thm1}
The following are equivalent. 
\begin{enumerate}[label=(A\arabic*)] 
    \item Either (C1) or (C2) holds;
    \item $\XX(p,q,p',q')$ is bounded;
    \item There exists a rough isometric embedding $R(p,q) \rightarrow R(p',q')$;
    \item There exists a bilipschitz embedding $\Z(p,q) \rightarrow \Z(p',q')$;
    \item There exists a quasiisometric embedding $HT(p,q) \rightarrow HT(p',q')$.
\end{enumerate}
\end{theorem}

Without much added difficulty, we also obtain the following. 

\begin{theorem} \label{thm2}
The following are equivalent. 
\begin{enumerate}[label=(B\arabic*)] 
    \item (C2) holds;
    \item $R(p,q)$ is rough isometric to $R(p',q')$;
    \item $\Z(p,q)$ is bilipschitz homeomorphic to $\Z(p',q')$;
    \item $HT(p,q)$ is quasiisometric to $HT(p',q')$. 
\end{enumerate}
\end{theorem}

\begin{remark*}
Some of the implications of \Cref{thm2} follow from \Cref{thm1}. If $R(p,q)$ is rough isometric to $R(p',q')$ then there exist rough isometric embeddings $R(p,q) \rightarrow R(p',q')$ and $R(p',q') \rightarrow R(p,q)$. Similar statements hold if $\Z(p,q)$ is bilipschitz homeomorphic to $\Z(p',q')$ or $HT(p,q)$ is quasiisometric to $HT(p',q')$. Hence, it follows from \Cref{thm1} that (B2), (B3), (B4) all imply (B1). 
\end{remark*}

The equivalence of (B1) and (B4), i.e. the quasiisometric classification of treebolic spaces, was conjectured by Woess (Question 2.15 \cite{WOESS}). 

\begin{corollary} \label{cor}
$HT(p,q)$ is quasiisometric to $HT(p',q')$ if and only if $\log(p)/\log(q) = \log(p')/\log(q')$ and $p, p'$ are powers of a common integer.
\end{corollary}

The treebolic space $HT(p,q)$ is an example of a \textit{horocyclic product} (for a definition see Section 2 of \cite{WOESS}). Other examples include the Diestel--Lieder graphs $\dl(p,q)$ (introduced in \cite{DL}) and the Lie groups $\sol(p,q)$. The quasiisometric classifications of the Diestel--Lieder graphs and the Sol groups were provided by Eskin, Fisher and Whyte \cite{EFWnote, EFW1}.

Recall that, when $m \in \N_{\geq 2}$, the treebolic space $HT(m,m)$ is quasiisometric to the solvable Baumslag-Solitar group $BS(1,m) = \langle a,t | tat^{-1} = a^m \rangle$ with some word metric (see, for example, Section 3 of \cite{FM}). It follows from \Cref{cor} that $BS(1,m)$ is quasiisometric to $BS(1,n)$ if and only if $m, n$ are powers of a common integer; this result was originally proved by Farb and Mosher \cite{FM}. Indeed, the proof in this paper that (A5) implies (A4) essentially follows from their work; the only meaningful change required is to the proof of Lemma 5.1 of their paper. \Cref{thm1} gives us the following result, proving that the solvable Baumslag--Solitar groups obey an even stronger rigidity.

\begin{corollary} \label{cor3}
There exists a quasiisometric embedding $BS(1,m) \rightarrow BS(1,n)$ if and only if $m,n$ are powers of a common integer.
\end{corollary}

\begin{remark*}
The content of \Cref{thm1} and \Cref{thm2} intersects with the work of Deng, Wen, Xiong and Xi \cite{BLE}. Indeed, the metric space $\Z(p,q)$ is a self-similar set satisfying the strong separation condition and is of Hausdorff dimension $\log(p) / \log(q)$. Theorem 1 of their paper then immediately gives us that (C1) implies (A4). Further, Theorem 2 of their paper gives us the following: if $\log(p) / \log(q) = \log(p') / \log(q')$ then (A4) holds if and only if (B3) holds. Now, Cooper proves in the appendix of \cite{FM} that if $\Z(m,m)$ is bilipschitz homeomorphic to $\Z(n,n)$ then $m,n$ are powers of a common integer. Consequently, once we have proved that (A5) implies (A4), \Cref{cor3} follows from \cite{BLE} and \cite{FM}. 
\end{remark*}

\Cref{thm1} tells us that unboundedness of the sequence $\XX(p,q,p',q')$ is an obstruction to the existence of a rough isometric embedding $R(p,q) \rightarrow R(p',q')$. It is possible to partially generalise this result to a far larger class of trees. 

\begin{itemize}
    \item Suppose we have sequences $(p_n)_{n \in \N}$, $(q_n)_{n \in \N}$ such that $p_n \in \N_{\geq 2}$ and $q_n \in \R_{>1}$. Suppose also that $p_n$ and $q_n$ are bounded sequences and that $\inf_n q_n > 1$. We can construct a rooted metric tree $R = R((p_n), (q_n))$ as follows. We start with a basepoint $b$. The basepoint $b$ has $p_1$ edges emanating from it of length $\log(q_1)$. The terminal vertices of these edges have $p_2$ edges emanating from them of length $\log(q_2)$. Then the terminal vertices of those edges have $p_3$ edges emanating from them of length $\log(q_3)$. Continuing in this way we will have constructed an infinite rooted tree $R((p_n),(q_n))$. We call a rooted tree constructed in this manner \textit{isotropic}. The regular rooted tree $R(p,q)$ is precisely the isotropic tree associated to the constant sequences $p_n = p$, $q_n = q$. 
\end{itemize}

\sloppy Given sequences $(p_n), (q_n), (p_n'), (q_n')$, one can define an integral sequence $\XX = \XX((p_n), (q_n), (p_n'), (q_n'))$ analogous to the sequence $\XX(p,q,p',q')$ defined above (see Section 3). We have the following.

\begin{theorem} \label{thm4}
Suppose we have a pair of isotropic trees $R((p_n), (q_n))$, $R((p_n'), (q_n'))$ such that $(p_n')$ is constant. If $\XX((p_n), (q_n), (p_n'), (q_n'))$ is unbounded then there does not exist a rough isometric embedding $R((p_n),(q_n)) \rightarrow R((p_n'),(q_n'))$.
\end{theorem}

\subsection{Structure of the paper}

Section 2 provides relevant details on the metric spaces $\Z(p,q)$ and $HT(p,q)$. Section 3 focusses on embedding isotropic trees; in it we will prove \Cref{thm4}, thereby proving that (A3) $\implies$ (A2). In Section 4, by studying the sequence $\XX(p,q,p',q')$, we will prove that (A2) $\implies$ (A1). In Section 5 we find rough isometric embeddings between regular rooted trees when either (C1) or (C2) hold; it covers the implications (A1) $\implies$ (A3) and (B1) $\implies$ (B2). Finally, Section 6 considers how embeddings of regular rooted trees $R(p,q)$, symbolic Cantor sets $\Z(p,q)$ and treebolic spaces $HT(p,q)$ all relate to each other; we will prove (A5) $\implies$ (A4) $\implies$ (A3) $\implies$ (A5), (B2) $\implies$ (B3) and (B2) $\implies$ (B4). 

\section{Preliminaries}

\subsection{Notation}

Given a metric space $X$, a subset $A \subset X$ and a constant $r \geq 0$, we use the notation $\mathcal{N}_r(A)$ to denote the \textit{$r$-neighbourhood of $A$}. That is, 
\[ \mathcal{N}_r(A) = \{ x \in X: d(x,A) < r \} \]
If we have another subset $B \subset X$, we use the notation $d_H(A,B)$ to denote the \textit{Hausdorff distance} between $A$ and $B$. This is defined to be the (possibly infinite) quantity
\[ d_H(A,B) = \max (\sup_{x \in A} d(x, B), \sup_{y \in B} d(y,A))\]

\subsection{Symbolic Cantor sets}

Let $\Z(p,q) = \{ 0, 1, ..., p-1 \}^{\N}$ be the space of infinite sequences on $p$ letters. We can put a metric $\rho$ on $\Z(p,q)$: given $(a_n), (b_n) \in \Z(p,q)$ set $\rho((a_n), (b_n)) = q^{-N}$ where $a_n = b_n$ for $n \leq N$ and $a_{N+1} \neq b_{N+1}$.

\subsection{Treebolic spaces}

Let $T(p,q)$ be the metric tree with all edges of length $\log(q)$ and with valency $p+1$ at every vertex. Let $\gamma: [0,\infty) \rightarrow T(p,q)$ be some geodesic ray based at a vertex $b$ of $T(p,q)$. The ray $\gamma$ induces a \textit{height function} $h: T(p,q) \rightarrow \R$ on $T(p,q)$ via
\[ h(x) = d(b,x)  - 2 \cdot \length(\gamma_x \cap \gamma)\]
where $\gamma_x$ is the geodesic segment from $b$ to $x$. Now consider the upper half plane model of the hyperbolic plane: $\hyp^2 = \{ (x,y) \in \R \times \R_{>0} \}$ with the metric $ds^2 = \frac{dx^2 + dy^2}{y^2}$. We can put a height function on this model of $\hyp^2$ via
\[ h(x,y) = \log(y) \]
Note that both these height functions coincide with the classical notion of a Busemann function (see Definition 8.17 \cite{BH}). By a \textit{horostrip} in $\hyp^2$ we mean a subset of the form $\{ (x,y) \in \hyp^2 : h(x,y) \in [a,b] \}$ for some real interval $[a,b]$. The space $HT(p,q)$ is formed by gluing horostrips onto every edge of $T(p,q)$ in a height-preserving manner. More precisely, let $HT(p,q) = T(p,q) \times \R$ as a set. Let $e = [v,w]$ be some edge of $T(p,q)$ such that $h(v) < h(w)$. We can put a metric on $e \times \R$ by identifying it with the horostrip
\[ \{ (x,y) \in \hyp^2 : h(x,y) \in [h(v), h(w)] \}\]
If we do this for every edge of $T(p,q)$ then we will have produced a metric on the whole of $HT(p,q)$ by taking the shortest path metric. The key to understanding the geometry of $HT(p,q)$ is the following fact. Let $\pi: HT(p,q) \rightarrow T(p,q)$ be the projection. Let $L: (-\infty, \infty) \rightarrow T(p,q)$ be a height-increasing bi-infinite geodesic in $T(p,q)$. Then $\pi^{-1}(L) \subset HT(p,q)$ is an isometrically embedded copy of $\hyp^2$. Hence we use the following terminology. A \textit{horocycle} in $HT(p,q)$ is a preimage $\pi^{-1}(x)$ where $x \in T(p,q)$. A \textit{branching horocycle} in $HT(p,q)$ is a preimage $\pi^{-1}(v)$ where $v$ is a vertex of $T(p,q)$. For more details on treebolic spaces see Section 2B of \cite{WOESS}.

\section{Embeddings of isotropic trees}

Let $R$ be some isotropic tree with basepoint $b$. We can put a height function $h: R \rightarrow \R_{\geq 0}$ on $R$ via $h(x) = d(b,x)$. We will also put an orientation on the edges of $R$: we orient the edges such that the initial vertex is closer to $b$ than the terminal vertex.

\begin{definition*}
Given some $x \in R$, we define its \textit{tree of descendants} $\tau(x)$ to be the subset of $R$ that is the union of $x$ and all points that can be reached from $x$ by paths that are in agreement with the orientation on $R$. Similarly, given a subset $B \subseteq R$, we can define
\[ \tau(B) = \bigcup_{x \in B} \tau(x) \]
\end{definition*}

Some notation.

\begin{itemize}
    \item $V(R)$ denotes the vertex set of $R$.
    \item Given a subset $B \subseteq R$ and some height $h \geq 0$ we use the notation $B \eval_h$ to denote the set of all points in $B$ of height $h$; this is always a finite set. 
    \item Suppose $x,y \in R$. Let $x \wedge y$ denote the point of $\{ z \in R: x \in \tau(z), y \in \tau(z) \}$ of maximal height.
\end{itemize}

We now give a further series of definitions.

\begin{definition*}
Let $A \geq 0$. We say that a map $f: R \rightarrow R'$ between two isotropic trees is \textit{$A$-coarsely height-preserving} if $\sup_{x \in R} \absval{h(x) - h(f(x))} \leq A$.
\end{definition*}

\begin{definition*}
Let $b,b'$ denote the basepoints of $R,R'$ respectively. A map $f: R \rightarrow R'$ is a \textit{waterfall map} if $f(b) = b'$ and $f \eval_{\gamma}$ is an isometry for all geodesic rays $\gamma$ based at $b$. 
\end{definition*}

\begin{definition*}
We say that a map $f: R \rightarrow R'$ between isotropic trees is \textit{order-preserving} if the following implication holds:
\[ y \in \tau(x) \implies f(y) \in \tau(f(x)) \]
That is, if $y$ descends from $x$ then $f(y)$ descends from $f(x)$.
\end{definition*}

Note that the following statements are equivalent: $f: R \rightarrow R$ is a waterfall map; $f$ is height-preserving and continuous; $f$ is height-preserving and order-preserving. 

\begin{definition*}
Let $A \geq 0$. A map $f: R \rightarrow R'$ is \textit{$A$-coarsely order-preserving} if $d(f(x),f(x) \wedge f(y)) \leq A$ for all $x,y \in R$ with $y \in \tau(x)$.
\end{definition*}

\begin{definition*}
Given a vertex $v \in V(R)$ we denote by $\mathcal{C}(v)$ the set of \textit{children} of $v$. That is, $\mathcal{C}(v) = \{ w_1, ..., w_p \}$ where $w_1, ..., w_p$ are the vertices connected by a single edge to $v$ such that $h(w_i) > h(v)$.
\end{definition*}

A waterfall map $f: R \rightarrow R'$ can be very far from any notion of injectivity. For example, it could map every geodesic ray based at $b$ onto one single geodesic ray based at $b'$. We will say that a waterfall map is \textit{distributive} if it makes as much effort as possible to be injective; if it divides itself evenly at the vertices of $R'$.

\begin{definition*}
Suppose $f: R \rightarrow R'$ is a waterfall map and suppose we have a vertex $w \in V(R')$ such that $f^{-1}(w)$ is non-empty. Write $f^{-1}(w) = \{ x_1, ..., x_m \}$. If the $x_i$ are vertices, then set $p = \absval{\mathcal{C}(x_i)}$, the number of edges emanating outwards from a single $x_i$. If the $x_i$ are not vertices, then set $p = 1$. Similarly, let $p' = \absval{\mathcal{C}(w)}$, the number of edges emanating outwards from the vertex $w$. Consider the subsets
\begin{align*}
B_\epsilon(w) &= \tau(w) \cap \{ x \in R : d(w,x) \leq \epsilon \} \\
B_\epsilon(x_i) &= \tau(x_i) \cap \{ x \in R' : d(x_i,x) \leq \epsilon \} \quad \quad (1 \leq i \leq m)
\end{align*}
and choose $\epsilon$ small enough that $B_\epsilon(w)$ is isometric to a star with $p'$ arms and $B_\epsilon(x_i)$ is isometric to a star with $p$ arms. Let $P = mp$. Since $f$ is a waterfall map, an arm emanating from some $x_i$ is mapped isometrically onto one of the $p'$ arms emanating from $w$. Thus, we get a map from a set of cardinality $P$ (the set of arms emanating from the $x_i$) to a set of cardinality $p'$ (the set of arms emanating from $w$). We say that $f$ is \textit{distributive} if at most $\lceil \frac{P}{p'} \rceil$ arms emanating from the $x_i$ can be mapped to the same arm emanating from $w$.
\end{definition*}

One could say that $f: R \rightarrow R'$ is distributive if it takes $b \mapsto b'$ and then lets $R$ \textit{expand like a gas} within $R'$. The notion of a distributive waterfall map should generalise naturally to the case of maps between $\R$-trees. Further, applying the result of Kerr \cite{KERR} which states that quasi--trees are rough isometric to $\R$-trees, there should be a well-defined notion of a distributive waterfall map between quasi-trees (i.e. one that is conjugate to a distributive waterfall map between two associated $\R$-trees). 

\begin{definition*}
Suppose we two isotropic trees $R = R((p_n), (q_n))$ and $R' = R((p_n'), (q_n'))$. Set $q_0 = q_0' = 1$. We have the sets of vertex heights
\[ \mathcal{H}(R) = \{ \sum_{i=0}^a \log(q_i) : a \in \N_{0} \} \quad \quad \mathcal{H}(R') = \{ \sum_{i=0}^b \log(q_i') : b \in \N_{0} \}\]
Now let $\mathcal{H} = \mathcal{H}(R) \cup \mathcal{H}(R')$ and write $\mathcal{H} = \{ h_0, h_1, h_2, ...\}$ such that $h_n < h_{n+1}$ for $n \in \N_0$. We can now define the infinite integral sequence $\XX = \XX((p_n),(q_n),(p_n'),(q_n')): \N_{\geq 0} \rightarrow \N$. Set $\XX_0 = 1$. If $\XX_n$ has already been defined then we set
\[
    \XX_{n+1} = 
    \begin{cases}
    p_{a+1}\XX_n, \quad &\textrm{if } h_{n} = \sum_{i=0}^a \log(q_i) \in \mathcal{H}(R) \setminus \mathcal{H}(R') \\
    \lceil \frac{\XX_n}{p_{b+1}'} \rceil, \quad &\textrm{if } h_{n} = \sum_{i=0}^b \log(q_i') \in \mathcal{H}(R') \setminus \mathcal{H}(R) \\
    \lceil \frac{p_{a+1} \XX_n}{p_{b+1}'} \rceil, \quad &\textrm{if } h_{n} = \sum_{i=0}^a \log(q_i) = \sum_{i=0}^b \log(q_i') \in \mathcal{H}(R) \cap \mathcal{H}(R')
    \end{cases}
\]
\end{definition*}

Note that in the case when the sequences have constant values $p_n = p, q_n = q, p_n' = p', q_n' = q'$ then we recover the sequence $\XX(p,q,p',q')$ described in the introduction. 

If we have a distributive waterfall map $f: R \rightarrow R'$ between isotropic trees, the sequence $\XX_n$ captures the maximum cardinality of preimages of points $x \in R'$. More precisely, $\XX_n$ takes the values of the piecewise constant function
\[ h \mapsto \max_{x \in R'\eval_{h}} \absval{f^{-1}(x)} \]
as $h$ increases.

\begin{remark*}
Another definition of the sequence $\XX_n$ can be found by observing that the progression of the sequence $\XX_n$ depends only on the order of the multiples of $\log(q)$ and $\log(q')$ within the sequence $(h_n)$ rather than on the values themselves. Hence, we could replace $\mathcal{H}$ with the set
\[ \mathcal{S} = \mathcal{S}(R) \cup \mathcal{S}(R') = \{ \prod_{i=0}^a q_i : a \in \N_{0} \} \cup \{ \prod_{i=0}^b q_i' : b \in \N_{0} \} = \{ s_0, s_1, s_2, ... \}\]
where $h_n = \log(s_n)$ and $\XX_n$ would not change.
\end{remark*}

A first step towards understanding these sequences would be to know the following.

\begin{question*}
Does the boundedness of the sequence $\XX((p_n),(q_n),(p_n'),(q_n'))$ depend upon its initial value? In other words, if we changed $\XX_0 = 1$ to $\XX_0 = x_0$ for some $x_0 \in \N$, could the sequence move from being bounded to unbounded or vice-versa?
\end{question*}

\begin{manualtheorem}{5}
Suppose we have a pair of isotropic trees $R((p_n), (q_n))$, $R((p_n'), (q_n'))$ such that $(p_n')$ is constant. If $\XX((p_n), (q_n), (p_n'), (q_n'))$ is unbounded then there does not exist a rough isometric embedding $R((p_n),(q_n)) \rightarrow R((p_n'),(q_n'))$.
\end{manualtheorem}

It is clear that if $\XX$ is unbounded then a waterfall map $R \rightarrow R'$ cannot be a rough isometric embedding since points $x \in R'$ would have preimages of arbitrarily large cardinality and of the same height as $x$. If we could show that an arbitrary rough isometric embedding $f: R \rightarrow R'$ is at bounded distance from a waterfall map then \Cref{thm4} would follow immediately and we would no longer need the requirement that $(p_n')$ is constant. 

\begin{question*}
Is every rough isometric embedding $R \rightarrow R'$ between isotropic trees at bounded distance from a waterfall map? 
\end{question*}

Recall that a waterfall map is a height-preserving and order-preserving map $R \rightarrow R'$. We will see that a rough isometric embedding $R \rightarrow R'$ is at bounded distance from a height-preserving map (\Cref{chpcop}) and at bounded distance from an order-preserving map (\Cref{prop1} and \Cref{prop2}). The curse is that this does not imply it is at bounded distance from a height-preserving \textit{and} order-preserving map $R \rightarrow R'$. However, the idea of the proof of \Cref{thm4} is to show that a rough isometric embedding $R \rightarrow R'$ is at bounded distance from a height-preserving and \textit{coarsely} order-preserving map $f: R \rightarrow R'$: $f$ can't send a pair of siblings too far away from each other - the siblings will become $k$'th cousins where $k$ depends only on the rough isometry constant of $f$. This will do.

\begin{lemma} \label{chpcop}
A rough isometric embedding $f: R \rightarrow R'$ between isotropic trees is coarsely height-preserving, coarsely order-preserving and at bounded distance from a height-preserving rough isometric embedding.
\end{lemma}

\begin{proof}
Suppose $f$ is an $A$-rough isometric embedding. Let $x \in R$. We have that
\[ \absval{h(f(x)) - h(x)} = \absval{d(b', f(x)) - d(b, x)} \leq d(b', f(b)) + \absval{d(f(b), f(x)) - d(b, x)} \leq B \]
where $B = d(b',f(b)) + A$. So $f$ is $B$-coarsely height-preserving. 

Now let $x,y \in R$ be such that $y \in \tau(x)$. We have that
\[ d(f(x),f(y)) \leq d(x,y) + A = h(y) - h(x) + A \leq h(f(y)) - h(f(x)) + 2B + A\]
Also,
\[ d(f(x),f(y)) = h(f(x)) + h(f(y)) - 2h(f(x) \wedge f(y)) \]
and so 
\[ h(f(x)) - h(f(x) \wedge f(y)) \leq \frac{2B + A}{2} \]
Hence $f$ is coarsely order-preserving.

We will now define a height-preserving map $g: R \rightarrow R'$ at bounded distance from $f$. We have three cases.
\begin{enumerate}
    \item If $h(f(x)) = h(x)$ then set $g(x) = f(x)$.
    \item If $h(f(x)) > h(x)$ then let $g(x)$ be the unique point of height $h(x)$ such that $f(x) \in \tau(g(x))$. Then $d(f(x), g(x)) = h(f(x)) - h(x) \leq B$.
    \item If $h(f(x)) < h(x)$ then let $g(x)$ be any element of $\tau(f(x))$ (there could be many) of height $h(x)$. Then $d(f(x), g(x)) = h(x) - h(f(x)) \leq B$.
\end{enumerate}
Thus we see that $g$ is height-preserving and $d(f, g) \leq B$. Since $g$ is at bounded distance from a rough isometric embedding, $g$ is itself a rough isometric embedding.
\end{proof}

\begin{proof}[Proof of Theorem 5]

Write $R = R((p_n), (q_n))$ and $R' = R((p_n'), (q_n'))$ and denote by $b,b'$ the basepoints of $R,R'$ respectively. We know that $(p_n')$ has constant value $p'$ for some $p' \in \N_{\geq 2}$. Let $P \in \N_{\geq 2}$ be an upper bound for $(p_n)$ and let $\epsilon, M > 0$ be such that $\log(q_n) \in [\epsilon, M]$. Similarly, let $\epsilon', M' > 0$ be such that $\log(q_n') \in [\epsilon', M']$. Suppose there exists an $A$-rough isometric embedding $f: R \rightarrow R'$ for some constant $A \geq 0$. By the lemma above, we can assume that $f$ is height-preserving. Increasing $A$ if necessary, we can also assume that $f$ is also $A$-coarsely order-preserving.\\

\bolden{Step 1: The plan.} Let $N \in \N$ be chosen such that $(N-1)\epsilon' \geq A$. The plan is to inductively define an infinite sequence of vertices $w_n \in V(R')$ that induce an infinite sequence of nested subtrees of $R'$
\[ \tau_0 = \tau(w_0) \supseteq \tau_1 = \tau(w_1) \supseteq \tau_2 = \tau(w_2) \supseteq ... \]
such that 
\begin{enumerate}[label=(T\arabic*)] 
    \item $\tau_n \eval_{h_n}$ has cardinality at most $(p')^N$;
    \item If $\tau_n \eval_{h_n}$ has cardinality less than $(p')^N$ then $w_n = b'$;
    \item $\summ_n := \absval{f^{-1}(\tau_n \eval_{h_n})} \geq \XX_n$.
\end{enumerate}
We need to make one more crucial observation.

\begin{claim*}
Suppose we have some $w \in V(R')$ inducing a subtree $\tau(w)$ and a pair of heights $h < H$ such that $\tau(w) \eval_{h}$ has cardinality at least $(p')^{N}$. Then
\[ f^{-1}(\tau(w) \eval_{H}) = \tau(f^{-1}(\tau(w) \eval_{h})) \eval_{H} \]
\end{claim*}
What is this saying? It says that (assuming the cardinality of $\tau(w) \eval_{h}$ is large) the points of $R$ which map into $\tau(w) \eval_{H}$ are precisely the descendants of those points of $R$ which map into $\tau(w) \eval_{h}$. 

\begin{proof}[Proof of claim]
Let $f(x) \in \tau(w) \eval_{h}$ and $y \in \tau(x) \eval_{H}$. Since $f$ is $A$-coarsely order-preserving, we have that $d(f(x) \wedge f(y), f(x)) \leq A \leq (N-1) \epsilon'$. Also, using the fact that $\tau(w) \eval_h$ has cardinality at least $(p')^N$, we know that $d(f(x), w) \geq (N-1)\epsilon'$. Hence, $f(x) \wedge f(y) \in \tau(w)$ and in particular $f(y) \in \tau(w)$. So
\[  \tau(f^{-1}(\tau(w) \eval_{h})) \eval_{H} \subseteq f^{-1}(\tau(w) \eval_{H}) \]
Now suppose $y \in R \eval_H$ descends from some $x \in R \eval_{h}$ (i.e. $y \in \tau(x)$) and suppose $x$ does not map into $\tau(w) \eval_h$. Clearly then $w \neq b'$ and $R' \eval_{h(w)}$ is a set of cardinality $k$ for some $k \geq 2$. There exists $w' \in R' \eval_{h(w)}$ such that $w' \neq w$ and $f(x) \in \tau(w') \eval_{h}$. By our work above, $f(y) \in \tau(w') \eval_{H}$ and so $f(y) \not \in \tau(w) \eval_{H}$ since $\tau(w)$ and $\tau(w')$ are disjoint. And so
\[ f^{-1}(\tau(w) \eval_{H}) \subseteq \tau(f^{-1}(\tau(w) \eval_{h})) \eval_{H} \qedhere \] 
\end{proof}

Property (T2) combined with the above claim implies that we have the equality 
\begin{equation} \label{eq}
f^{-1}(\tau_n \eval_{h_{n+1}}) = \tau(f^{-1}(\tau_n \eval_{h_n})) \eval_{h_{n+1}}
\end{equation}
for each of the (yet to be defined) subtrees $\tau_n$. \\

\bolden{Step 2: The base case.} Set $w_0 = b'$. Then $\tau_0 = \tau(b') = R'$. We need to verify that (T1), (T2) and (T3) hold. Since $h_0 = 0$, $\tau_0 \eval_{h_0}$ has cardinality $1$ which is evidently no greater than $(p')^N$. Certainly $\tau_0 \eval_{h_0}$ has cardinality less than $(p')^N$ but (T2) holds since $w_0 = b'$. Finally note that 
\[ \summ_0 = 1 = \XX_0\]

\bolden{Step 3: The induction.} Suppose $\tau_n$ has been defined. We will now define $\tau_{n+1}$. We have three cases to consider.

\begin{enumerate}
    \item Suppose $h_n \in \mathcal{H}(R) \setminus \mathcal{H}(R')$. That is, $h_n$ is a vertex height for $R$ but not $R'$. In this case we set $\tau_{n+1} = \tau_n$. Then $\tau_{n+1} \eval_{h_{n+1}} = \tau_n \eval_{h_{n+1}}$ will have the same cardinality as $\tau_n \eval_{h_n}$ since $R'$ has no vertices with height in the interval $[h_n, h_{n+1})$. So $\tau_{n+1}$ satisfies properties (T1) and (T2). Further, by \eqref{eq} we have that
    \[ \summ_{n+1} = p_{a+1} \summ_n \]
    where $h_n = \sum_{i=0}^a \log(q_i)$. By induction, we can assume that $\summ_n \geq \XX_n$ and so
    \[ \summ_{n+1} = p_{a+1} \summ_n \geq p_{a+1} \XX_n = \XX_{n+1} \]
    \item We divide into two subcases.
    \begin{enumerate}
        \item Suppose $h_n \in \mathcal{H}(R') \setminus \mathcal{H}(R)$ and $\tau_n \eval_{h_n}$ has cardinality less than $(p')^N$ (and so in fact it has cardinality at most $(p')^{N-1}$). By induction, $w_n = b'$. Note also that the cardinality of $\tau_n \eval_{h_{n+1}}$ is $p'$ times larger than the cardinality of $\tau_n \eval_{h_n}$. Set $\tau_{n+1} = \tau_n$. Then it is clear that (T1) and (T2) hold. Also,
        \[ \summ_{n+1} = \summ_n \geq \XX_n \geq \lceil \frac{\XX_n}{p'} \rceil = \XX_{n+1} \]
        and so (T3) holds.
        \item Suppose $h_n \in \mathcal{H}(R') \setminus \mathcal{H}(R)$ and $\tau_n \eval_{h_n}$ has cardinality $(p')^N$. In this case, $\tau_n \eval_{h_{n+1}}$ has cardinality $(p')^{N+1}$. Let $\{ v_1, ..., v_{p'}\} = \mathcal{C}(w_n)$. $\tau_{n} \eval_{h_{n+1}}$ can be partitioned into $p'$ sets of cardinality $(p')^N$ each:
        \[ \tau_{n} \eval_{h_{n+1}} = \bigsqcup_{i=1}^{p'} \tau(v_i) \eval_{h_{n+1}} \]
        And so, using \eqref{eq}, we get a partition
        \[ \tau(f^{-1}(\tau_n \eval_{h_n})) \eval_{h_{n+1}} = f^{-1}(\tau_n \eval_{h_{n+1}}) = \bigsqcup_{i=1}^{p'} f^{-1}(\tau(v_i) \eval_{h_{n+1}}) \]
        The cardinality of the above set is precisely $\summ_n$ and so there must exist some $1 \leq j \leq p'$ such that 
        \[ \absval{f^{-1}(\tau(v_j) \eval_{h_{n+1}})} \geq \lceil \frac{\summ_n}{p'} \rceil \]
        Set $w_{n+1} = v_j$. It is clear that (T1) and (T2) hold and 
        \[ \summ_{n+1} = \absval{f^{-1}(\tau(v_j) \eval_{h_{n+1}})} \geq \lceil \frac{\summ_n}{p'} \rceil \geq \lceil \frac{\XX_n}{p'} \rceil = \XX_{n+1} \]
        and so (T3) holds. 
    \end{enumerate}
    \item Again, we divide into two subcases.
    \begin{enumerate}
        \item Suppose $h_n \in \mathcal{H}(R) \cap \mathcal{H}(R')$ and $\tau_n \eval_{h_n}$ has cardinality less than $(p')^N$. As before we set $\tau_{n+1} = \tau_n$ and (T1) and (T2) hold. We have that
        \[ \summ_{n+1} = p_{a+1} \summ_n \geq p_{a+1} \XX_n \geq \lceil \frac{p_{a+1} \XX_n}{p'} \rceil = \XX_{n+1} \]
        where $h_n = \sum_{i=0}^a \log(q_i)$. So (T3) holds.
        \item Finally suppose that $h_n \in \mathcal{H}(R) \cap \mathcal{H}(R')$ and $\tau_n \eval_{h_n}$ has cardinality $(p')^N$. Again we write $\{ v_1, ..., v_{p'}\} = \mathcal{C}(w_n)$ and we have
        \[ \tau(f^{-1}(\tau_n \eval_{h_n})) \eval_{h_{n+1}} = f^{-1}(\tau_n \eval_{h_{n+1}}) = \bigsqcup_{i=1}^{p'} f^{-1}(\tau(v_i) \eval_{h_{n+1}}) \]
        The cardinality of the above set is $p_{a+1} \summ_n$ where $h_n = \sum_{i=0}^a \log(q_i)$. Hence, there exists some $1 \leq j \leq p'$ such that 
        \[ \absval{f^{-1}(\tau(v_j) \eval_{h_{n+1}})} \geq \lceil \frac{p_{a+1} \summ_n}{p'} \rceil \]
        and we set $w_{n+1} = v_j$. (T1) and (T2) hold as before and
        \[ \summ_{n+1} = \absval{f^{-1}(\tau(v_j) \eval_{h_{n+1}})} \geq \lceil \frac{p_{a+1} \summ_n}{p'} \rceil \geq \lceil \frac{p_{a+1} \XX_n}{p'} \rceil = \XX_{n+1} \]
        and so (T3) holds as well. 
    \end{enumerate}
\end{enumerate}

\bolden{Step 4: The conclusion.} If $\XX_n$ is unbounded then $\summ_n$ must be unbounded too. But $\summ_n$ is the preimage under $f$ of a set of cardinality at most $(p')^N$ and so we can conclude preimages of single points under $f$ can have arbitrarily large cardinality. Since $f$ is height-preserving, preimages of single points must all have the same height in $R$, i.e. they are contained in some set $R \eval_{h}$. Further, since $f$ is an $A$-rough isometry, they must all be contained in some ball of radius $A$. However, there is a fixed upper bound on the cardinality of balls of radius $A$ in a level set $R \eval_{h}$. Hence if $\XX_n$ is unbounded there cannot exist a rough isometric embedding of $R$ into $R'$. 
\end{proof}

Naturally, one hopes that the requirement that $(p_n')$ is constant is unnecessary. 

\begin{conjecture*}
Suppose we have two isotropic trees $R = R((p_n), (q_n))$ and $R' = R((p_n'), (q_n'))$. If $\XX((p_n), (q_n), (p_n'), (q_n'))$ is unbounded then there does not exist a rough isometric embedding of $R$ into $R'$. 
\end{conjecture*}

One might suspect the converse holds: that if $\XX_n$ is bounded then there exists a rough isometric embedding of $R$ into $R'$. This seems plausible since if $\XX_n$ is bounded then a distributive waterfall map from $R$ into $R'$ would have bounded point preimages. 

\begin{question*}
If $\XX_n((p_n), (q_n), (p_n'),(q_n'))$ is bounded, does there exist a rough isometric embedding $R((p_n),(q_n)) \rightarrow R((p_n'),(q_n'))$?
\end{question*}

However, as the following example shows, we can find pairs of isotropic trees for which $\XX_n$ is bounded and for which \textit{none} of the distributive waterfall maps are rough isometric embeddings. 

\begin{example*}
Let $R = R((p_n),(q_n))$ where $p_1 = 6$, $q_1 = 2$ and $p_n = 3$, $q_n = 4$ for $n \geq 2$. Let $R' = R((p_n'),(q_n'))$ where $p_n' = 3$, $q_n' = 4$ (equivalently $R' = R(3,4)$). Let $b,b'$ denote the basepoints of $R,R'$ respectively. The associated sequence $\XX_n((p_n), (q_n), (p_n'), (q_n'))$ is bounded; indeed, $\XX_n \leq 6$. Suppose that $f$ is a distributive waterfall map of $R$ into $R'$. We will find two distinct geodesic rays, $\gamma_B$ and $\gamma_Y$, based at $b$ which have the same image under $f$. We first need to establish a combinatorial fact. \\

\textit{Suppose we have $6$ objects such that $3$ are coloured blue and $3$ are coloured yellow. If we want to partition the objects into $3$ groups of size $2$ then there must be a group which contains a blue and a yellow.} \\

$6$ arms emanate from $b$ which are divided between the $3$ arms emanating from $b'$. Thus, pairs of arms emanating from $b$ are crushed together by $f$. Consider a fixed pair of arms emanating from $b$ that are crushed together under $f$. We will say that descendants of the first arm are \textit{blue} and descendants of the second arm are \textit{yellow}. Consider the terminal vertex $w_1 \in R'$ of the edge onto which the blue and yellow arm are crushed. The preimage of this vertex consists of $6$ points, $3$ blue and $3$ yellow. By the statement above, we know that a blue arm and a yellow arm must be crushed onto the same edge emanating from $w_1$. Let $w_2$ be the terminal vertex of this edge. Again, we know that the preimage of $w_2$ consists of $3$ blue and $3$ yellow points and we know that a blue arm and a yellow arm must be crushed onto some edge emanating from $w_2$. Continuing in this way, we get a sequence of vertices $b, w_1, w_2, w_3, ...$ such that if $\eta$ is the geodesic which joins them all together then $f^{-1}(\eta)$ contains a blue geodesic ray $\gamma_B$ and a yellow geodesic ray $\gamma_Y$. Points on $\gamma_B$ and $\gamma_Y$ of the same height have the same image under $f$ even though they may be arbitrarily far away in $R$. So $f$ is not a rough isometric embedding.

Note, however, that there does nonetheless exist a rough isometric embedding $R \rightarrow R'$; if we crush $\{ x \in R : h(x) \leq \log(2) \}$ to a point then we get $6$ copies of $R(3,4)$ glued together; if we crush $\{ y \in R' : h(y) \leq 2\log(4)\}$ to a point then we get $9$ copies of $R(3,4)$ glued together.

\end{example*}

\section{The sequence $\XX(p,q,p',q')$}

The sequence $\XX(2,2,3,3)$ can be understood as follows. We start with the non-negative number line $\R_{\geq 0}$ (a long pebbly beach perhaps). We begin with a single pebble in our hand. We then imagine walking along the number line playing the following game as we go: each time we pass a power of $2$, we multiply the amount of pebbles in our possession by $2$; each time we pass a power of $3$, we divide our pebbles into $3$ groups as evenly as possible and then keep only one of the larger groups. The question is whether there exists some fixed upper bound on the number of pebbles in our possession as we play this game.

\begin{figure}[h]
\centering
\includegraphics[width=1\textwidth]{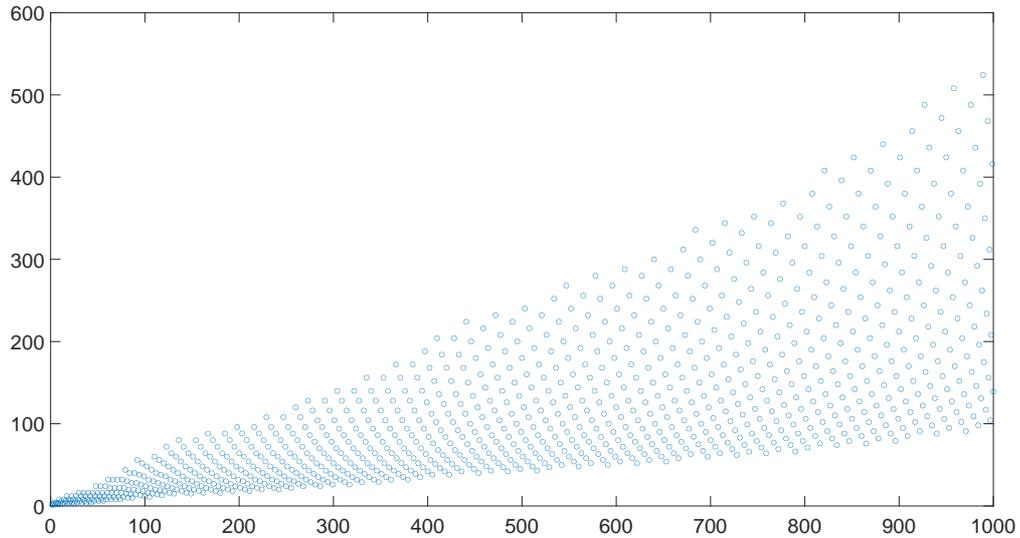} 
\caption{The sequence $\XX(2,2,3,3)$. Plotted up to $n = 1000$.}
\label{fig1}
\end{figure}

In order to prove that (A2) $\implies$ (A1), we need the following lemma.

\begin{lemma} \label{lemma}
Suppose that $\log(q')/\log(q)$ is irrational. Let $a \log(q)$ be a multiple of $\log(q)$ and let $b \log(q')$ be the least multiple of $\log(q')$ that is greater than $a \log(q)$. Then there exists $a' > a$ such that, if $b' \log(q')$ is the least multiple of $\log(q')$ greater than $a' \log(q)$, we have
\[ b'\log(q') - a' \log(q) < b \log(q') - a \log(q) \]
\end{lemma}

\begin{proof}
Write $d = b \log(q') - a \log(q)$. Consider the metric space $X = \R / (\log(q) \Z)$. Let $\theta: X \rightarrow X$ be rotation by $\log(q')$
\[ [x] \xmapsto{\theta} [x + \log(q')] \]
We can think of $X$ as a circle with circumference $\log(q)$ such that as $x$ increases, $[x]$ moves anticlockwise around the circle. We have a bilipschitz map
\[ X = \R / (\log(q) \Z) \xrightarrow{\eta} \R / \Z \]
given by
\[ [x] \mapsto \Big[\frac{x}{\log(q)}\Big] \]
$\theta$ is conjugate under $\eta$ to the map $\theta' = \eta \theta \eta^{-1}: \R / \Z \rightarrow \R / \Z$
\[ [x] \xmapsto{\theta'} \Big[x + \frac{\log(q')}{\log(q)}\Big] \]
which has dense orbits since $\frac{\log(q')}{\log(q)}$ is irrational. Consequently, $\theta$ has dense orbits. In particular, we can find arbitrarily large $b'$ such that the anticlockwise distance from $[0]$ to $\theta^{(b')}([0]) = [b' \log(q')]$ is less than $d$. If we choose $a'$ such that $a' \log(q)$ is the largest multiple of $\log(q)$ less than $b' \log(q')$, and if we choose $b'$ large enough that $a' > a$, then $a'$ and $b'$ satisfy the desired properties. 
\end{proof}

\begin{theorem} \label{thm}
If $\XX(p,p',q,q')$ is bounded then either (C1) or (C2) holds.
\end{theorem}

\begin{proof}
First, suppose that $\frac{\log(p)}{\log(q)} > \frac{\log(p')}{\log(q')}$. Let $\beta = \frac{\log(q')}{\log(q)}$. So $p^{\beta} > p'$. Consider the real sequence defined by $\mathcal{Y}_0 = 1$ and 
\[
    \mathcal{Y}_{n+1} = 
    \begin{cases}
    p\mathcal{Y}_n, \quad &\textrm{if } h_{n} \in \mathcal{H}(R) \setminus \mathcal{H}(R') \\
    \frac{\mathcal{Y}_n}{p'}, \quad &\textrm{if } h_{n} \in \mathcal{H}(R') \setminus \mathcal{H}(R) \\
    \frac{p\mathcal{Y}_n}{p'}, \quad &\textrm{if } h_{n} \in \mathcal{H}(R) \cap \mathcal{H}(R')
    
    \end{cases}
\]
Suppose $h_n = a\log(q)$. We have that
\[ \mathcal{Y}_{n+1} = p^{a+1} \Big(\frac{1}{p'}\Big)^{\lfloor \frac{a\log(q)}{\log(q')} \rfloor + 1} \geq p^{a+1} \Big(\frac{1}{p'}\Big)^{ \frac{a\log(q)}{\log(q')} + 1} = \frac{p}{p'} \Big(\frac{p^\beta}{p'}\Big)^{\frac{a}{\beta}} \]
The right-hand side of the above goes to infinity as $a$ goes to infinity and hence $\mathcal{Y}_n$ is unbounded. Thus $\XX_n$ is also unbounded since $\XX_n \geq \mathcal{Y}_n$ for all $n \in \N_0$. 

Now suppose that $\frac{\log(p)}{\log(q)} = \frac{\log(p')}{\log(q')}$ and $p, p'$ are not powers of a common integer. This implies that $\log(q')/\log(q)$ is irrational. Let $h_l = a\log(q)$ be some arbitrary element of $\mathcal{H}(R)$. Let $b \log(q')$ be the least multiple of $\log(q')$ greater than $a \log(q)$. By \Cref{lemma} we can find some $a' > a$ such that if $b' \log(q')$ is the least multiple of $\log(q')$ greater $a' \log(q)$ then
\[ b'\log(q') - a'\log(q) < b\log(q') - a\log(q) \]
which rearranges to
\[ b'-b < (a' - a)\frac{\log(q)}{\log(q')}\]
If $h_m = a'\log(q) \in \mathcal{H}(R)$ then
\[ \XX_m \geq p^{a' - a}\Big(\frac{1}{p'}\Big)^{b'-b}\XX_l > p^{a'-a}\Big(\frac{1}{p'}\Big)^{(a' - a)\frac{\log(q)}{\log(q')}}\XX_l = \XX_l \]
So $\XX_m > \XX_l$. But both $\XX_m$ and $\XX_l$ are integral and so in fact $\XX_m \geq \XX_l + 1$. We have shown that given $l \in \N_0$ such that $h_l \in \mathcal{H}(R)$ we can find $m > l$ such that $h_m \in \mathcal{H}(R)$ and $\XX_m \geq \XX_l + 1$. Repeating this process gives arbitrarily large values of $\XX_n$ and so the sequence is unbounded. 
\end{proof}

\section{Embeddings of regular rooted trees}

\begin{proposition} \label{thm7}
If (C2) holds then $R(p,q)$ is rough isometric to $R(p',q')$. 
\end{proposition}

\begin{proof}

Write $R = R(p,q)$ and $R' = R(p',q')$. Let $b,b'$ denote the basepoints of $R,R'$ respectively. First, suppose that there exists some $s \in \N$ such that $p' = p^s$ and $q' = q^s$. Let $V_s \subset R$ consist of all vertices in $R$ of height $ks \log(q)$ for some $k \in \N_0$. Let $V'$ denote the vertex set of $R'$. We will define a map $f: V_s \rightarrow V'$ inductively. Set $f(b) = b'$. Suppose $v \in V_s$ and $f(v)$ has already been defined. The sets $\tau(v) \eval_{h(v) + s\log(q)}$ and $\mathcal{C}(f(v))$ both have $p^s$ elements. Define $f$ on $\tau(v) \eval_{h(v) + s\log(q)}$ by choosing any bijection between these two sets. Then $f$ is a rough isometry from $V_s$ to $V'$. If we extend $f$ using the nearest point projection $R \rightarrow V_s$ and the inclusion $V' \rightarrow R'$ then we get a rough isometry $R \rightarrow R'$. 

In the general case, write $p = r^s$, $p' = r^t$ where $r \in \N_{\geq 2}$. Then $q' = q^{\frac{t}{s}}$. By the above, we know that $R = R(p,q)$ is rough isometric to $R(r, q^{\frac{1}{s}})$ which is rough isometric to $R(r^t, q^{\frac{t}{s}}) = R(p', q') = R'$. So $R$ is rough isometric to $R'$. 
\end{proof}

\begin{proposition}
If (C1) holds then there exists a rough isometric embedding $R(p,q) \rightarrow R(p',q')$.
\end{proposition}

\begin{proof}
We will first show that if $p = p'$ and $q > q'$ then there exists a rough isometric embedding $f: R(p,q) \rightarrow R(p',q')$. We will first define $f$ on the vertex set $V$ of $R(p,q)$. Set $f(b) = b'$. If $v \in V$ and $f(v)$ is already defined then we define $f$ on $\mathcal{C}(v)$ as follows. Write $\mathcal{C}(v) = \{ v_1, ..., v_p \}$. We know that $\tau(f(v)) \eval_{h(v) + \log(q)}$ has cardinality $p^k$ for some $k \in \N$. Map $\mathcal{C}(v)$ into this set injectively. If we precompose this map $V \rightarrow R(p',q')$ with the nearest point projection $R(p,q) \rightarrow V$ then we get a rough isometric embedding $R(p,q) \rightarrow R(p',q')$.

Now suppose the general case holds, i.e. $\frac{\log(p)}{\log(q)} < \frac{\log(p')}{\log(q')}$. Say $\frac{\log(p')}{\log(q')} = (1+ \xi) \frac{\log(p)}{\log(q)}$ where $\xi > 0$. Evidently, we can find $a, b \in \N$ such that
\begin{equation} \label{iq}
\log(p) \leq \frac{b}{a} \log(p') < (1 + \xi)\log(p)
\end{equation}
Then
\[ \frac{b \log(p')}{a \log(q)} < (1 + \xi)\frac{\log(p)}{\log(q)} = \frac{\log(p')}{\log(q')} \]
and so $a \log(q) > b \log(q')$, or, equivalently, $q^a > (q')^b$. By \Cref{thm7}, we know that $R(p,q)$ is rough isometric to $R(p^a, q^a)$. Further, it is clear there exists an isometric embedding of $R(p^a, q^a)$ into $R((p')^b, q^a)$ since, by \eqref{iq}, $(p')^b \geq p^a$. Then, by our work above, there exists a rough isometric embedding of $R((p')^b, q^a)$ into $R((p')^b, (q')^b$ since $q^a > (q')^b$. Finally, we know that $R((p')^b), (q')^b)$ is rough isometric to $R(p', q')$. So there exists a rough isometric embedding of $R(p,q)$ into $R(p',q')$.
\end{proof}

\section{Equivalence of the embeddings}

\subsection{Generalising Farb--Mosher}

In this section we will prove that (A5) $\implies$ (A4). 

\begin{definition*}
Let $\Q(p,q)$ denote the set of bi-infinite sequences $(a_n)_{n \in \Z}$ on the alphabet $\{0, ..., p-1\}$ such that $a_n = 0$ for all sufficiently small $n$. We can put a metric on $\Q(p,q)$ by setting $\rho((a_n), (b_n)) = q^{-N}$ where $a_n = b_n$ for $n \leq N$ and $a_{N+1} \neq b_{N+1}$. A \textit{clone} of $\Q(p,q)$ is a subset consisting of all words beginning with some fixed word $w$. Every clone in $\Q(p,q)$ is bilipschitz homeomorphic to $\Z(p,q)$. 
\end{definition*}

Let $\partial^u HT(p,q)$ denote the set of hyperbolic planes in $HT(p,q)$. Via the projection map $\pi: HT(p,q) \rightarrow T(p,q)$, the set $\partial^u HT(p,q)$ is in bijection with the set of height-increasing bi-infinite geodesics in $T(p,q)$. This, in turn, can be seen to be in bijection with $\Q(p,q)$. We can put a metric on $\partial^u HT(p,q)$ as follows. If $Q_1, Q_2 \in \partial^u HT(p,q)$ then set $d(Q_1, Q_2) = e^{-h(\sigma)}$ where $\sigma$ is the horocycle $\partial(Q_1 \cap Q_2)$. With this metric $\partial^u HT(p,q)$ is isometric to $\Q(p,q)$. Let $m,n \in \N_{\geq 2}$ and suppose $f: HT(m,m) \rightarrow HT(n,n)$ is a quasiisometry. In the paper of Farb and Mosher \cite{FM} (Proposition 4.1) they prove that if $Q \in \partial^u HT(m,m)$ then there exists a unique $Q' \in \partial^u HT(n,n)$ such that $d_H(f(Q), Q') < \infty$ and they write $f^u(Q) = Q'$. They also prove that, with respect to the metrics on $\partial^u HT(m,m)$ and $\partial^u HT(n,n)$, $f^u$ is a bilipschitz homeomorphism (Theorem 6.1 of \cite{FM}). So a quasiisometry $f: HT(m,m) \rightarrow HT(n,n)$ induces a bilipschitz homeomorphism $f^u: \Q(m,m) \rightarrow \Q(n,n)$. Let $C \subset \Q(m,m)$ be some clone. Since $f^u$ is a bilipschitz homeomorphism and $C$ has finite diameter, $f^u(C)$ also has finite diameter and hence is contained in some clone $C' \subset \Q(n,n)$. Thus, since $C$ is bilipschitz homeomorphic to $\Z(m,m)$ and $C'$ is bilipschitz homeomorphic to $\Z(n,n)$, we have an induced bilipschitz embedding $\Z(m,m) \rightarrow \Z(n,n)$. In summary, Farb and Mosher prove that a quasiisometry $HT(m,m) \rightarrow HT(n,n)$ induces a bilipschitz embedding $\Z(m,m) \rightarrow \Z(n,n)$. 

We would like to generalise this result in two ways. First, we would like to show that it holds if we no longer require $f$ to be coarse surjective. That is, we assume only that $f$ is a quasiisometric embedding. Second, we would like to show that it holds if we replace $HT(m,m)$ with $HT(p,q)$ and $HT(n,n)$ with $HT(p,q)$. Put concisely: we would like to prove that a quasiisometric embedding $HT(p,q) \rightarrow HT(p',q')$ induces a bilipschitz embedding $\Z(p,q) \rightarrow \Z(p',q')$. Therefore, we need to analyse - up until the conclusion of the proof of Theorem 6.1 - all the instances when Farb and Mosher use the coarse surjectivity of $f$ and all the instances when they use the fact that $p = q$ and $p' = q'$. One of these questions is easy to answer: the proof never once uses the fact that $p = q$ and $p' = q'$ so we can replace $HT(m,m)$ with $HT(p,q)$ and $HT(n,n)$ with $HT(p',q')$. The coarse surjectivity of $f$ is used once, in the proof of Lemma 5.1. Thus we will need to reprove this lemma without ever using coarse surjectivity. This is the content of \Cref{relem} below. 

\begin{remark*}
It is important to note that Theorem 7.2 of the Farb--Mosher paper is not actually proved as it is written. Theorem 7.2 says that if there is a bilipschitz embedding $\Z(m,m) \rightarrow \Z(n,n)$ then $m,n$ are powers of a common integer. However, in the paper Cooper actually proves Corollary 10.11: that if there is a bilipschitz embedding $\Z(m,m) \rightarrow \Z(n,n)$ \textit{onto a clopen} then $m,n$ are powers of a common integer. Indeed, the proof relies on the fact that the image is a clopen set in $\Z(n,n)$. They can assume this extra condition since they have proved that $f^u: \Q(m,m) \rightarrow \Q(n,n)$ is a bilipschitz homeomorphism (which follows from the coarse surjectivity of $f$) and hence the image of $C \subset \Q(m,m)$ will be clopen in $C' \subset \Q(n,n)$. In our case, when $f$ is not coarse surjective, we know only that $f^u$ is a bilipschitz embedding. However, it follows from the work in this paper that Theorem 7.2 is nonetheless true. 
\end{remark*}

So suppose we have some $(K,C)$-quasiisometric embedding $f: HT(p,q) \rightarrow HT(p',q')$. For simplicity of notation, write $HT(p,q) = X$, $HT(p',q') = X'$. A generalisation of Proposition 4.1 in Farb--Mosher implies that there exists a constant $A \geq 0$, only depending on $K$ and $C$, such that for all $Q \in \partial^u X$ there exists some unique $Q' \in \partial^u X'$ such that $d_H(f(Q), Q') \leq A$. We write $f^u(Q) = Q'$. We can use the map $f^u$ to define a map $\theta_f: T(p,q) \rightarrow T(p',q')$ as follows (here we are identifying $T(p,q)$ with the set of horocycles in $X = HT(p,q)$ and similarly for $T(p',q')$). Suppose we have some horocycle $\sigma \subset HT(p,q)$. Let $\mu$ be the first branching horocycle above $\sigma$; say $\mu = \partial(Q_1 \cap Q_2)$ where $Q_1, Q_2 \in \partial^u X$. Let $Q_i' \in \partial^u X'$ have bounded Hausdorff distance from $f(Q_i)$. Then we set $\theta_f(\sigma) = \partial(Q_1' \cap Q_2')$. The following lemma removes the coarse surjectivity requirement of Lemma 5.1 of Farb--Mosher. 

\begin{lemma} \label{relem}
Given $K \geq 1$, $C \geq 0$, there exists a constant $\lambda \geq 0$ such that if $f:X \rightarrow X'$ is a $(K,C)$-quasiisometric embedding, then for each horocycle $\sigma \subset X$ we have
\[ d_H(f(\sigma), \theta_f(\sigma)) \leq \lambda \]
\end{lemma}

\begin{proof}
In the proof of Lemma 5.1 of Farb--Mosher, they prove that there exists a constant $\lambda_1 = \lambda_1(K,C) \geq 0$ such that for each horocycle $\sigma \subset X$ we have
\[ f(\sigma) \subset \mathcal{N}_{\lambda_1}(\sigma') \]
where $\sigma' = \theta_f(\sigma)$. In order to show that there exists some $\lambda_2 = \lambda_2(K,C) \geq 0$ such that $\sigma' \subset \mathcal{N}_{\lambda_2}(f(\sigma))$, Farb and Mosher use the coarse inverse of $f$ which we no longer have. However, we can apply an alternative argument. 

Let $z \in \sigma'$. Let $\mu$ be the branching horocycle above $\sigma$ such that $\theta_f(\sigma) = \theta_f(\mu)$. Let $Q_1, Q_2 \subset X$ be the two hyperbolic planes in $X$ such that $\mu = \partial(Q_1 \cap Q_2)$ and let $Q_1', Q_2' \subset X'$ be the two hyperbolic planes in $X'$ such that $\sigma' = \partial(Q_1' \cap Q_2')$. Recall that for $i = 1,2$ we have
\[ d_H(f(Q_i), Q_i') \leq A \]
where $A$ only depends on $K,C$. Let $z_1$ be the unique point of $Q_1' \setminus Q_2'$ satisfying $d(z_1, z) = d(z_1, \sigma') = d(z_1, Q_2') = 2A + 1$. Similarly, Let $z_2$ be the unique point of $Q_2' \setminus Q_1'$ satisfying $d(z_2, z) = d(z_2, \sigma') = d(z_2, Q_1') = 2A + 1$. We know there exists some $x_i \in Q_i$ such that $d(f(x_i), z_i) \leq A$. Thus, $d(f(x_1), Q_2') \geq A+1$ and $d(f(x_2), Q_1') \geq A+1$. Hence, $x_1 \in Q_1 \setminus Q_2$ and $x_2 \in Q_2 \setminus Q_1$. Also,
\[ d(f(x_1), f(x_2)) \leq 2(2A + 1) + 2A = 6A + 2\]
and so
\[ d(x_1, x_2) \leq K(6A + 2) + C\]
which implies that 
\[ d(x_i, \mu) \leq K(6A + 2) + C\]
Let $w_i$ be the unique point of $\sigma$ realising $d(\sigma, x_i)$. Then $d(w_i, x_i) \leq d(x_i, \mu) + \log(q)$. We know that
\[ d(f(w_i), f(x_i)) \leq K(K(6A + 2) + C + \log(q)) + C\]
and so
\[ d(f(w_i), z) \leq K(K(6A + 2) + C + \log(q)) + C + (3A + 1) \]
and so we are done.
\end{proof}

The rest of Farb and Mosher's argument goes through almost verbatim - one only has to change $m$ and $n$ to $p,q,p',q'$ as appropriate. It follows that if $f: HT(p,q) \rightarrow HT(p',q')$ is a quasiisometric embedding then $f^u: \Q(p,q) \rightarrow \Q(p',q')$ is a bilipschitz embedding. Let $C \subset \Q(p,q)$ be some clone. We know that $f^u(C)$ has finite diameter and hence is contained in some clone $C' \subset \Q(p',q')$. We know that $C$ and $C'$ are bilipschitz homeomorphic to $\Z(p,q)$ and $\Z(p',q')$ respectively. So we have proved the following. 

\begin{proposition}
If there exists a quasiisometric embedding $HT(p,q) \rightarrow HT(p',q')$ then there exists a bilipschitz embedding $\Z(p,q) \rightarrow \Z(p',q')$. 
\end{proposition}

\subsection{Functors}

In this section we will prove that (A4) $\iff$ (A3) and (B3) $\iff$ (B2). Much of this section is simply an application of the ideas contained in the paper of Bonk and Schramm \cite{BS} on the functors $\partial$ and $\con$ between Gromov-hyperbolic metric spaces and their boundary. However, for the purposes of this paper, it is simpler to `start from scratch' as opposed to translating their more general work over to the needs of our specialised situation. 

Suppose for every vertex $v$ of $R(p,q)$ we arbitrarily label the edges emanating from $v$ with the letters $\{ 0, ..., p-1 \}$. This edge labelling provides a natural identification of $\Z(p,q)$ with the set of geodesic rays based at the basepoint $b$ of $R(p,q)$. Throughout this section we make this identification. We begin with definitions.

\begin{definition*}
Suppose we have a pair of maps $f, g: R(p,q) \rightarrow R(p',q')$ between regular rooted trees. We say that $f$ and $g$ are \textit{roughly equivalent} if $d(f,g) < \infty$. We denote by $[f]$ the rough equivalence class of some map $f$ between regular rooted trees. 
\end{definition*}

\begin{definition*}
Let $\mathcal{C}$ be the category consisting of regular rooted trees $R(p,q)$ and rough isometric embeddings between them \textit{up to rough equivalence}.
\end{definition*}

\begin{definition*}
Let $\mathcal{D}$ denote the category of symbolic Cantor sets $\Z(p,q)$ and bilipschitz embeddings between them. 
\end{definition*}

We have the following theorem whose proof will be the content of this section.

\begin{theorem} \label{isothm}
There exists a category isomorphism $\partial: \mathcal{C} \rightarrow \mathcal{D}$ with inverse functor $\Delta: \mathcal{D} \rightarrow \mathcal{C}$.
\end{theorem}

Evidently, \Cref{isothm} implies that (A4) $\iff$ (A3) and (B3) $\iff$ (B2). We begin by defining $\partial: \mathcal{C} \rightarrow \mathcal{D}$. 

\begin{definition*}
We set $\partial R(p,q) = \Z(p,q)$. Let $f: R(p,q) \rightarrow R(p',q')$ be a rough isometric embedding. Let $\gamma \in \Z(p,q)$. It follows from the Morse Lemma that there exists some $\gamma' \in \Z(p',q')$ with $d_H(f(\gamma), \gamma') < \infty$. Indeed, such a $\gamma'$ must be unique since distinct geodesic rays in $R(p',q')$ have infinite Hausdorff distance. We set $\partial [f] (\gamma) = \gamma'$. 
\end{definition*}

One needs to check that $\partial [f]$ is well-defined: if $f$ is roughly equivalent to $g$ then $d_H(f(\gamma), g(\gamma)) < \infty$ and so $\gamma'$ must be the same in both cases. It is also simple to check that $\partial$ is functorial. Some notation.

\begin{itemize}
    \item Suppose $\gamma, \eta \in \Z(p,q)$. Let $\gamma \wedge \eta$ denote the point of $\gamma \cap \eta \subset R(p,q)$ of maximal height. 
\end{itemize}

\begin{proposition} \label{prop}
Let $f: R(p,q) \rightarrow R(p',q')$ be a rough isometric embedding. Then $\partial [f]: \Z(p,q) \rightarrow \Z(p',q')$ is a bilipschitz embedding. 
\end{proposition}

\begin{proof}
Write $g = \partial [f]$. Let $\rho$ denote the metric on $\Z(p,q)$ and let $\rho'$ denote the metric on $\Z(p,q)$. Let $\gamma, \eta \in \Z(p,q)$ and write $\gamma' = g(\gamma), \eta' = g(\eta)$. By \Cref{chpcop}, we know that $f$ is at bounded distance from a height-preserving rough isometric embedding. So, without loss of generality, we can assume that $f$ is height-preserving. We know the following
\[ \rho(\gamma, \eta) = q^{-\frac{h(\gamma \wedge \eta)}{\log(q)}} = e^{-h(\gamma \wedge \eta)} \]
Similarly, 
\[ \rho'(\gamma', \eta') = (q')^{-\frac{h(\gamma' \wedge \eta')}{\log(q')}} = e^{-h(\gamma' \wedge \eta')} \]
So we just need to show that $\absval{h(\gamma \wedge \eta) - h(\gamma' \wedge \eta')} \leq C$ where $C$ is some constant that does not depend on $\gamma, \eta$. We know that there exists some constant $B \geq 0$ such that $d_H(f(\gamma), \gamma') \leq B$ and $d_H(f(\eta), \eta') \leq B$. Let $x \in \gamma$, $y \in \eta$ be such that $d(f(x), \gamma' \wedge \eta') \leq B$ and $d(f(y), \gamma' \wedge \eta') \leq B$. Then $d(f(x), f(y)) \leq 2B$ and so $d(x,y) \leq 2B + A$ where $A$ is the rough isometry constant of $f$. Consequently, $d(x, \gamma \cap \eta) \leq 2B + A$ and so 
\[ h(\gamma \wedge \eta) \geq h(x) - (2B + A) \geq h(\gamma' \wedge \eta') - (3B + A) \]
Now let $w \in \gamma'$, $z \in \eta'$ be such that $d(w, f(\gamma \wedge \eta)) \leq B$ and $d(z, f(\gamma \wedge \eta)) \leq B$. Then $d(w,z) \leq 2B$ and so $d(w, \gamma' \cap \eta') \leq 2B$. Hence
\[ h(\gamma' \wedge \eta') \geq h(w) - 2B \geq h(\gamma \wedge \eta) - 3B \]
and we are done.
\end{proof}

Before defining $\Delta: \mathcal{D} \rightarrow \mathcal{C}$, we need some more terminology.

\begin{itemize}
    \item By a \textit{clone} in $\Z(p,q)$ we mean a subset $C \subseteq \Z(p,q)$ consisting of all words beginning with some fixed word. Suppose we have some subset $S \subseteq \Z(p,q)$. The \textit{minimal clone} containing $S$ is the clone containing $S$ of minimal diameter. \item Using an arbitrary edge labelling of $R(p,q)$, we can identify a vertex $v$ of $R(p,q)$ with a finite word $u$ in the alphabet $\{ 0, ..., p-1\}$. Denote by $C(v)$ the clone consisting of all words beginning with $u$. In this manner, we have a bijection between the vertices of $R(p,q)$ and the clones of $\Z(p,q)$. 
\end{itemize}

\begin{definition*}
We set $\Delta \Z(p,q) = R(p,q)$. Suppose we have a bilipschitz embedding $g: \Z(p,q) \rightarrow \Z(p',q')$. Let $v \in V(R)$ be some vertex of $R(p,q)$. We have that $gC(v) \subseteq \Z(p',q')$ is contained in some minimal clone $C(w) \subseteq \Z(p',q')$ where $w \in V(R')$. We set $f(v) = w$. Thus, we have constructed a map $f: V(R) \rightarrow R(p',q')$. Precomposition with the nearest point projection map gives us a map $f: R(p,q) \rightarrow R(p',q')$. We set $[f] = \Delta g$. 
\end{definition*}

\begin{proposition} \label{prop1}
Let $g: \Z(p,q) \rightarrow \Z(p',q')$ be a $\lambda$-bilipschitz embedding. Let $[f] = \Delta g$ where $f: V(R) \rightarrow R(p',q')$ is the map defined above. Then $f$ is an order-preserving rough isometric embedding with associated constant depending only on $\lambda$. 
\end{proposition}

\begin{proof}
Suppose $v, v'$ are vertices of $R = R(p,q)$ such that $v'$ descends from $v$. We want to show that $f(v')$ descends from $f(v)$. Write $w = f(v)$ and $w' = f(v')$. Since $C(v') \subseteq C(v)$ we have that $gC(v') \subseteq gC(v)$. Therefore $gC(v') \subseteq C(w)$. Thus, by the minimality of $C(w')$, we have $C(w') \subseteq C(w)$. So $w'$ descends from $w$. Hence, $f$ is order-preserving.

We will now show that $f$ is a rough isometric embedding. Let $v, v' \in V(R)$ and, as before, write $w = f(v)$, $w' = f(v')$. We have three cases to consider: $v'$ descends from $v$, $v$ descends from $v'$ and neither descends from the other. Suppose $v'$ descends from $v$. Say, $d(b,v) = k \log(q_i)$ and $d(b,v') = l\log(q)$ where $k,l$ are non-negative integers with $k \leq l$. Then $d(v,v') = (l-k)\log(q)$. We also know that $\diam(C(v)) = q^{-k}$ and $\diam(C(v')) = q^{-l}$. Thus
\[ d(v, v') = \log(\frac{\diam(C(v))}{\diam(C(v'))})\]
Similarly, since $w'$ descends from $w$, we have that 
\[ d(w, w') = \log(\frac{\diam(C(w))}{\diam(C(w'))})\]
By the minimality of $C(w')$ and $C(w)$ we know that $\diam(C(w')) = \diam(gC(v'))$ and $\diam(C(w)) = \diam(gC(v))$. Therefore
\[ d(w, w') = \log(\frac{\diam(gC(v))}{\diam(gC(v'))}) \leq \log(\lambda^2 \frac{\diam(C(v))}{\diam(C(v'))}) = d(v,v') + 2\log(\lambda)\]
and similarly
\[ d(w,w') \geq \log(\frac{1}{\lambda^2} \frac{\diam(C(v))}{\diam(C(v'))}) = d(v,v') - 2\log(\lambda)\]
Here we have used the fact that $g$ is $\lambda$-bilipschitz and so only changes diameters by factors in the interval $[\lambda^{-1}, \lambda]$. So we are done in the case when $v'$ descends from $v$. By symmetry this also proves the case when $v$ descends from $v'$. Now suppose that neither descends from the other. Let $W = f(v \wedge v')$. Some notation: if $A,B \in \R_{>0}$ we write $A \sim_\lambda B$ if there exists some positive constant $K = K(\lambda)$ such that $\frac{B}{K} \leq A \leq KB$. Further, if $x,y \in \R$, we write $x \approx_\lambda y$ if there exists some real constant $D = D(\lambda)$ such that $\absval{x - y} \leq D$. We have that
\begin{align*}
\diam(C(w \wedge w')) &= \diam(gC(v) \cup gC(v')) \\
&\sim_\lambda \diam(C(v) \cup C(v')) \\
&= \diam(C(v \wedge v')) \\
&\sim_\lambda \diam(C(W))
\end{align*}
And so $\diam(C(W)) \sim_\lambda \diam(C(w \wedge w'))$. Since $w \wedge w'$ descends from $W$ this implies that 
\[  d(w \wedge w', W) = \log(\frac{\diam(C(W))}{\diam(C(w \wedge w'))}) \approx_\lambda 0\]
Also, note that our above work on the case when $v'$ descends from $v$ implies that $d(w,W) \approx_{\lambda} d(v,v \wedge v')$ and $d(W,w') \approx_{\lambda} d(v \wedge v', v')$. Therefore
\begin{align*}
d(w, w') &= d(w, w \wedge w') + d(w \wedge w', w') \\
&\approx_{\lambda} d(w, W) + d(W, w') \\
&\approx_{\lambda} d(v, v \wedge v') + d(v \wedge v', v') \\
&= d(v,v')
\end{align*}
Hence, $f$ is an $A$-rough isometric embedding for some $A \geq 0$ where $A$ only depends on $\lambda$.
\end{proof}

The following proposition proves \Cref{isothm}.

\begin{proposition} \label{prop2}
$\partial: \mathcal{C} \rightarrow \mathcal{D}$ and $\Delta: \mathcal{D} \rightarrow \mathcal{C}$ are mutually inverse functors.
\end{proposition}

\begin{proof}
Let $g: \Z(p,q) \rightarrow \Z(p',q')$ be a $\lambda$-bilipschitz embedding and let $[f] = \Delta g$. We will first show that $\partial [f] = g$. Let $(a_n) \in \Z(p,q)$, let $\gamma$ be the corresponding geodesic ray and suppose that $(v_n)$ is the sequence of vertices traversed by $\gamma$. Write $w_n = f(v_n)$. We know that $(a_n) \in C(v_n)$ and $C(v_n) \supseteq C(v_{n+1})$ for all $n$. Hence $g((a_n)) \in gC(v_n)$ for all $n$. Further, $gC(v_n) \subseteq C(w_n)$ and so $g((a_n)) \in C(w_n)$ for all $n$. So we have
\[ C(w_1) \supseteq C(w_2) \supseteq C(w_3) \supseteq ... \ni g((a_n))\]
where we have used the fact that $f$ preserves order. We know that $\diam(C(v_n)) \rightarrow 0$ and so $\diam(C(w_n)) = \diam(gC(v_n)) \rightarrow 0$ since $g$ is a bilipschitz embedding. Therefore $g((a_n))$ is contained in a decreasing sequence of clones; there can only be one element contained in all of them. Let $\gamma'$ be the geodesic ray joining together the vertices $w_n$. Since $f(v_n) = w_n$, it follows from the Morse lemma that $d_H(f(\gamma), \gamma') < \infty$ and so $\partial [f](\gamma) = \gamma'$. But $\gamma' \in C(w_n)$ for all $n$ and so $\gamma' = g((a_n))$. Hence $\partial \Delta g = g$. 

We will now prove that if we have an $A$-rough isometric embedding $f: R(p,q) \rightarrow R(p',q')$ then $\Delta \partial [f] = [f]$. By \Cref{chpcop}, we can assume that $f$ is height-preserving. Let $v \in V(R)$. Suppose $C(w)$ is the minimal clone such that $(\partial [f]) C(v) \subseteq C(w)$. Then $\Delta (\partial [f]) (v) = w$. We want to show that $w$ is at some uniformly bounded distance from $f(v)$. By the minimality of $w$, there must exist $\gamma, \eta \in C(v)$ such that $\gamma \wedge \eta = v$ and $\gamma' \wedge \eta' = w$ where $\gamma' = \partial [f](\gamma)$, $\eta' = \partial [f] (\eta)$. By the same argument as in \Cref{prop}, it follows that $\absval{h(v) - h(w)}$ is uniformly bounded. Let $z \in \gamma'$ be such that $d(z, f(v)) \leq B$. Then clearly $\absval{h(z) - h(v)} \leq B$. So
\[ d(f(v), w) \leq d(f(v), z) + d(z, w) \leq B + \absval{h(z) - h(w)} \leq 2B + \absval{h(v) - h(w)}\]
which is uniformly bounded.

Finally, observe that $\Delta$ is a functor because $\partial$ is
\[ \Delta(g' \circ g) = \Delta(\partial\Delta g' \circ \partial\Delta g) = \Delta\partial(\Delta g' \circ \Delta g) = \Delta g' \circ \Delta g \qedhere\] 
\end{proof}

\subsection{From trees to treebolic}

In this section we will prove that (A3) $\implies$ (A5) and (B2) $\implies$ (B4). Recall that $T(p,q)$ is the infinite regular tree with vertex valency $p+1$ and edge length $\log(q)$. We can view $T(p,q)$ as two copies of $R(p,q)$ connected together at the basepoints by a line of length $\log(q)$. With this in mind, write $T(p,q) = R_1 \sqcup R_2 \sqcup L$ where $R_1,R_2$ are isometric to $R(p,q)$ and $L$ is the open line of length $\log(q)$. We let $b_1$ denote the basepoint of $R_1$ and let $b_2$ denote the basepoint of $R_2$. We can also write $T(p',q') = R_1' \sqcup R_2' \sqcup L'$ where $R_1', R_2',L',b_1',b_2'$ are defined similarly.

\begin{proposition} \label{treethm}
A rough isometric embedding $R(p,q) \rightarrow R(p',q')$ induces a rough isometric embedding $T(p,q) \rightarrow T(p',q')$. Further, if $R(p,q)$ is rough isometric to $R(p',q')$ then $T(p,q)$ is rough isometric to $T(p',q')$.
\end{proposition}

\begin{proof}
Suppose we have an $A-$rough isometric embedding $f: R(p,q) \rightarrow R(p',q')$. Let $V(T)$ be the vertex set of $T(p,q)$, let $V_1$ be the vertex set of $R_1$, let $V_2$ be the vertex set of $R_2$ and let $V$ be the vertex set of $R(p,q)$. We define a map $F: V(T) \rightarrow T(p',q')$ as follows. If $v \in V_1$ then set $F(v) = f(v)$ where we have identified $V_1$ with $V$ and $R(p',q')$ with $R_1'$. If $v \in V_2$ then set $F(v) = f(v)$ where we have now identified $V_2$ with $V$ and $R(p',q')$ with $R_2'$. We need to prove that $F$ is a rough isometric embedding. If $v,w \in V_1$ or $v,w \in V_2$, then $F$ roughly preserves the distance between $v$ and $w$ since $f$ does. Write $\kappa = d(b' f(b))$. If $v \in V_1$ and $w \in V_2$ then we have that 
\begin{align*}
d(F(v), F(w)) &= d(F(v), b_1') + \log(q') + d(b_2', F(w)) \\
&= d(f(v), b') + \log(q') + d(b', f(w)) \\
&\leq d(f(v), f(b)) + \log(q') + d(f(b), f(w)) + 2\kappa \\
&\leq d(v,b) + \log(q) + d(b,w) + 2\kappa + 2A\\
&= d(v,w) + 2\kappa + 2A
\end{align*}
Similarly, $d(F(v), F(w)) \geq d(v,w) - 2\kappa - 2A$. So $F$ is a rough isometric embedding. Note that if $f$ is a rough isometry then so is $F$. 
\end{proof}

\begin{proposition} \label{hthm}
Suppose we have some fixed height functions on $T(p,q)$ and $T(p',q')$ determined by the geodesic rays $\gamma$ and $\gamma'$ respectively. If there exists an $A$-rough isometric embedding $f: T(p,q) \rightarrow T(p',q')$ then there exists a coarsely height-preserving $A$-rough isometric embedding $g: T(p,q) \rightarrow T(p',q')$. Further, if $f$ is a rough isometry then so is $g$. 
\end{proposition}

\begin{proof}
Suppose $\gamma$ has basepoint $b$ and $\gamma'$ has basepoint $b'$. Let $b''$ be the vertex of $T(p',q')$ closest to $f(b)$. By the Morse lemma, there exists a unique geodesic ray $\gamma''$ based at $b''$ such that $d_H(\gamma'', f(\gamma)) \leq B$ where $B \geq 0$. Let $\psi: T(p',q') \rightarrow T(p',q')$ be some isometry taking $\gamma''$ to $\gamma'$. We claim that $g = \psi f$ is coarsely height-preserving. Let $x \in T(p,q)$. We have that
\[ d(b,x) \leq d(f(b),f(x)) + A \leq d(b'', f(x)) + A + \log(q') = d(b', g(x)) + A + \log(q') \]
Similarly, $d(b,x) \geq d(b',g(x)) - A - \log(q')$. Let $\eta$ denote the geodesic segment from $b$ to $x$. Let $\xi$ denote the geodesic segment from $b'$ to $g(x)$. We have that $d_H(\xi, g(\eta)) \leq C$ for some uniform constant $C \geq 0$ by the Morse Lemma. Let $\gamma \wedge \eta$ denote the point of $\gamma \cap \eta$ of furthest distance from $b$ and let $\gamma' \wedge \xi$ be defined similarly. Since $g(\gamma)$ is uniformly close to $\gamma'$ and $g(\eta)$ is uniformly close to $\xi$, identical arguments to those in the proof of \Cref{prop} show that the difference between $d(b, \gamma \wedge \eta)$ and $d(b', \gamma' \wedge \xi)$ is uniformly bounded. Since
\[ h(x) = d(b,x) - 2 \cdot d(b, \gamma \wedge \eta)\]
and
\[ h(g(x)) = d(b',g(x)) - 2 \cdot d(b', \gamma' \wedge \xi)\]
it follows that $g$ is coarsely height-preserving. 
\end{proof}

Recall that topologically $HT(p,q) = T(p,q) \times \R$. Let $\pi: HT(p,q) \rightarrow T(p,q)$ be the projection onto the first factor and let $\zeta: HT(p,q) \rightarrow \R$ be the projection onto the second factor. Given $x \in HT(p,q)$, we say that $\zeta(x)$ is the \textit{depth} of $x$. Further, we choose a height function on $HT(p,q)$ that agrees with the height function on $T(p,q)$: i.e. $h(\pi x) = h(x)$ for all $x \in HT(p,q)$. 

Suppose we have a coarsely height-preserving rough isometric embedding $f: T(p,q) \rightarrow T(p',q')$. There exists a unique map $\hat{f}: HT(p,q) \rightarrow HT(p',q')$ satisfying $\pi \hat{f} x = f \pi x$ and $\zeta \hat{f} x = \zeta x$ for all $x \in HT(p,q)$. We call $\hat{f}$ the \textit{horocyclic extension} of $f$. We want to show that $\hat{f}$ is a quasiisometric embedding. Consider the function
\[ D: HT(p,q) \times HT(p,q) \rightarrow \R_{\geq 0} \]
given by 
\[ D(x,y) = \min\{d(x,y'), d(y, x')\}\]
where $y'$ is the unique point with $\pi y' = \pi x$, $\zeta y' = \zeta y$ and $x'$ is the unique point with $\pi x' = \pi y$, $\zeta x' = \zeta x$. The key observations are that for $x,y \in HT(p,q)$
\begin{itemize}
    \item $d(x,y) \geq d(\pi x, \pi y)$ and $d(x,y) \geq D(x,y)$;
    \item $d(x,y) \leq d(\pi x, \pi y) + D(x,y)$.
\end{itemize}

\begin{proposition} \label{qiethm}
If $f: T(p,q) \rightarrow T(p',q')$ is a $B$-coarsely height-preserving $A$-rough isometric embedding then the horocyclic extension $\hat{f}: HT(p,q) \rightarrow HT(p',q')$ is a $(2, A + 2B)$-quasiisometric embedding. Further, if $f$ is a rough isometry then $\hat{f}$ is a quasiisometry. 
\end{proposition}

\begin{proof}
First observe that $\hat{f}$ is also $B$-coarsely height-preserving:
\[ \absval{h(\hat{f}x) - h(x)} = \absval{h(\pi\hat{f}x) - h(\pi x)} = \absval{h(f\pi x) - h(\pi x)} \leq B\]

\begin{claim*}
$D$ is coarsely preserved by $\hat{f}$. That is, $\absval{D(x,y) - D(\hat{f}x, \hat{f}y)} \leq 2B$.
\end{claim*}

\begin{proof}[Proof of claim]
Let $x,y \in T(p,q)$. Suppose $\absval{h(x) - h(\hat{f}x)} = \delta \leq B$. There exists a quadrilateral in $\hyp^2$ with vertices $v_1, v_2, v_3, v_4$ such that $d(v_1, v_2) = d(x,y')$, $d(v_3, v_4) = d(\hat{f}x, \hat{f}y)$ and $d(v_2, v_3) = d(v_4, v_1) = \delta$. Hence
\[ \absval{d(x,y') - d(\hat{f}x, \hat{f}y')} \leq 2B\]
Similarly, 
\[ \absval{d(y,x') - d(\hat{f}y, \hat{f}x')} \leq 2B\]
and so $\absval{D(x,y) - D(\hat{f}x, \hat{f}y)} \leq 2B$. 
\end{proof}

We have that
\begin{align*}
d(x,y) &\leq d(\pi x, \pi y) + D(x,y) \\
&\leq d(f \pi x, f \pi y) + A + D(\hat{f}x, \hat{f}y) + 2B \\
&\leq d(\pi \hat{f} x, \pi \hat{f} y) + A + d(\hat{f}x, \hat{f}y) + 2B \\
&\leq 2 d(\hat{f}x, \hat{f}y) + A + 2B
\end{align*}
and
\begin{align*}
d(\hat{f}x,\hat{f}y) &\leq d(\pi \hat{f}x, \pi \hat{f}y) + D(\hat{f}x,\hat{f}y) \\
&\leq d(f \pi x, f \pi y) + D(x, y) + 2B \\
&\leq d(\pi x, \pi y) + A + d(x, y) + 2B \\
&\leq 2 d(x, y) + A + 2B
\end{align*}
and so $\hat{f}$ is a $(2, A+2B)$-quasiisometric embedding. Suppose now that $f$ is also $A$-coarse surjective and let $z \in HT(p',q')$. Let $x \in T(p,q)$ be such that $d(f(x), \pi z) \leq A$. Let $y \in HT(p,q)$ satisfy $\pi y = x$ and $\zeta y = \zeta z$. Then $d(\hat{f}(y), z) \leq A$ since $\hat{f}(y)$ and $z$ have the same depth. So $\hat{f}$ is also coarse-surjective. 
\end{proof}

The implications (A3) $\implies$ (A5) and (B2) $\implies$ (B4) follow from \Cref{treethm}, \Cref{hthm} and \Cref{qiethm}.

\section*{Acknowledgements}

I would like to thank my supervisor Cornelia Dru\textcommabelow{t}u for reading early drafts of this paper and providing helpful advice.

\bibliographystyle{ieeetr}
\bibliography{ref.bib}

\end{document}